\documentclass[11 pt]{article}

\usepackage{latexsym, amsmath, amssymb, longtable, booktabs,amscd,microtype,booktabs,amsthm,fancyhdr,amsfonts}

\usepackage{etoolbox}
\usepackage[square]{natbib}
\usepackage[affil-it]{authblk}
\usepackage{enumerate}
\usepackage{tikz}
\usetikzlibrary{matrix,arrows,decorations.pathmorphing}
\usepackage{hyperref}
\bibpunct{[}{]}{,}{n}{}{;}

\numberwithin{equation}{section}
\newtheorem{thm}{Theorem}[subsection]

\newtheorem{prop}[thm]{Proposition}
\newtheorem{lemma}[thm]{Lemma}
\theoremstyle{definition}

\theoremstyle{remark}
\newtheorem{rmk}[thm]{\textbf{Remark}}

\newtheorem{claim}[thm]{\textbf{Claim}}

\theoremstyle{definition}
\newtheorem{dfn}{Definition}[subsection]

\theoremstyle{remark}

\theoremstyle{remark}

\makeatletter
\def\imod#1{\allowbreak\mkern10mu({\operator@font mod}\,\,#1)}

\makeatother

\title{\textbf{A Detailed Study of Kirchhoff-type Critical Elliptic Equations and $p$-Sub-Laplacian Operators within the Heisenberg Group $\mathcal{H}_{n}$ Framework}}
\author{SUBHAM DE}
\affil{Department of Mathematics, Indian Institute of Technology, Delhi, India.}
\affil{Email: \textbf{mas227132@iitd.ac.in}\\ Website: \textbf{www.sites.google.com/view/subhamde}}

\date{Dated: \today}

\begin{document}
	
	\maketitle
	\thispagestyle{empty}
	
	\begin{abstract}
		This article presents a comprehensive study of \textit{Kirchhoff-type Critical Elliptic Equations} involving $p$-sub-Laplacian Operators on the \textit{Heisenberg Group} $\mathcal{H}_{n}$. It delves into the mathematical framework of Heisenberg Group, and explores their Spectral Properties. A significant focus is on the existence and multiplicity of solutions under various conditions, leveraging concepts like the \textit{Mountain Pass Theorem}. This work not only contributes to the theoretical understanding of such groups but also has implications in fields like Quantum Mechanics and Geometric Group Theory.
		\texttt{\textbf{2020 MSC: }}\texttt{Primary  35A15, 35A20, 35B38, 35B65, 35J35, 35D30, 35H10, 35J70 }.\\
		\hspace*{50pt}\texttt{Secondary 35-02, 35A01, 35A09 }.
	\end{abstract}

	\pagenumbering{arabic}
	
	\newpage
	\tableofcontents

	\pagestyle{fancy}
	
	\fancyhead[LO,RE]{\markright}
	\lfoot[]{Subham De}
	\rfoot[]{IIT Delhi, India}

	\section{Sub-Laplacian On The Heisenberg Group}
	\subsection{Definition $\&$ Construction}
	Suppose, we consider the identification of $\mathbb{R}^{2}$ with $\mathbb{C}$ via the map,
	\begin{align*}
		\mathbb{R}^{2}\longrightarrow \mathbb{C}\\
		(x,y)\mapsto z=x+iy
	\end{align*}
	Then, we can interpret, $\mathcal{H}_{3}=\mathbb{C}\times \mathbb{R}$, where, $\mathcal{H}_{3}$ is the \textit{Heisenberg Group} defined on $3$ parameters.
	As observed before, $\mathcal{H}_{3}$ is a \textit{non-commutative} and a \textit{Unimodular Lie Group} on which the \textit{Haar Measure} is equal to the usual \textit{Lebesgue Measure} $dzdt$.\par 
	A priori denoting the Lie Algebra associated to $\mathcal{H}_{3}$ as $\mathfrak{h}$, consisting of all left invariant vector fields on the same, we can in fact opt for a basis of $\mathfrak{h}$ as $\{X,Y,T\}$, where,
	\begin{align}\label{16}
		X=\partial_{y_{1}}-2y_{2}\partial_{\tau}\mbox{ , \hspace{20pt}}Y=\partial_{y_{2}}+2y_{1}\partial_{\tau}\mbox{ , \hspace{20pt}}T=4\partial_{\tau}
	\end{align}
	\begin{dfn}\label{dfn16}
		(Sub-Laplacian)  The \textbf{sub-Laplacian} $\mathcal{L}$ on $\mathcal{H}_{3}$ is defined by,
		\begin{align}\label{21}
			\mathcal{L}=-\left(X^{2}+Y^{2}\right)
		\end{align}
	\end{dfn}
	We further introduce the following notations corresponding to the partial differential operators on $\mathbb{C}$ as,
	\begin{align*}
		\frac{\partial}{\partial z}=\frac{\partial}{\partial y_{1}}-i\frac{\partial}{\partial y_{2}}
	\end{align*}
	\begin{align*}
		\frac{\partial}{\partial \overline{z}}=\frac{\partial}{\partial y_{1}}+i\frac{\partial}{\partial y_{2}}
	\end{align*}
	Thus it only suffices to study the vector fields $Z$ and $\overline{Z}$ on $\mathcal{H}_{3}$ given by,
	\begin{align} \label{19}
		Z=X-iY=\frac{\partial}{\partial z}-2i\overline{z}\frac{\partial}{\partial \tau}
	\end{align}
	\begin{align} \label{20}
		\overline{Z}=X+iY=\frac{\partial}{\partial \overline{z}}+2iz\frac{\partial}{\partial \tau}
	\end{align}
	Important to note that, $\overline{Z}$ is also well-known as the \textit{Hans Lewy Operator} \cite{19}, which eventually defies \textit{local solvability} on $\mathbb{R}^{3}$, and,
	\begin{align}\label{22}
		\mathcal{L}=-\frac{1}{2}\left(Z\overline{Z}+\overline{Z}Z\right)
	\end{align}
	We can further compute,
	\begin{align*}
		\mathcal{L}=-\left(\left(\frac{\partial}{\partial y_{1}}-2y_{2}\frac{\partial}{\partial \tau}\right)^{2}+\left(\frac{\partial}{\partial y_{2}}+2y_{1}\frac{\partial}{\partial \tau}\right)^{2}\right)\\
		=-\Delta -4\left(y_{1}^{2}+y_{2}^{2}\right)\frac{\partial^{2}}{\partial \tau^{2}}+4\left(y_{2}\frac{\partial}{\partial y_{1}}-y_{1}\frac{\partial}{\partial y_{2}}\right)\frac{\partial}{\partial \tau}
	\end{align*}
	provided, $\Delta=\frac{\partial^{2}}{\partial y_{1}^{2}}+\frac{\partial^{2}}{\partial y_{2}^{2}}$.\par 
	Subsequently, the symbol $\sigma(\mathcal{L})$ of $\mathcal{L}$ can be derived as follows,
	\begin{align}
		\sigma(\mathcal{L})(y_{1},y_{2},\tau;\xi,\eta,\gamma)=\left(\xi-2y_{2}\gamma\right)^{2}+\left(\eta+2y_{1}\gamma\right)^{2}
	\end{align}
	For every $(y_{1},y_{2},\tau)\mbox{ }, \mbox{ }(\xi,\eta,\gamma)\in \mathcal{H}_{3}$.
	\begin{rmk}
		$\mathcal{L}$ is \textbf{Nowhere Elliptic} on $\mathbb{R}^{3}$.
	\end{rmk}
	A priori from the fact that, $[X,Y]=T$, a theorem by \textit{H$\ddot{o}$rmander} \cite[Theorem~1.1]{18} enables us to conclude that, $\mathcal{L}$ is indeed \textit{Hypoelliptic}.
	\subsection{Twisted Laplacians}
	For $\tau\in \mathbb{R}\setminus \{0\}$, let, $Z_{\tau}$ and $\overline{Z}_{\tau}$ be \textit{partial differential operators given by,
		\begin{align*}
			Z_{\tau}=\frac{\partial}{\partial z}-2\overline{z} \tau \\
			\overline{Z}_{\tau}=\frac{\partial}{\partial \overline{z}}+2z\tau 
	\end{align*}}
	Subsequently, the \textbf{Twisted Laplacian} $L_{\tau}$ is defined as,
	\begin{align} \label{23}
		L_{\tau}=-\frac{1}{2}\left(Z_{\tau}\overline{Z}_{\tau}+\overline{Z}_{\tau}Z_{\tau}\right)
	\end{align}
	To be more explicit, we can write,
	\begin{align*}
		L_{\tau}=-\frac{1}{2}\left(\left(\frac{\partial}{\partial z}-2\overline{z} \tau\right)\left(\frac{\partial}{\partial \overline{z}}+2z\tau\right)+\left(\frac{\partial}{\partial \overline{z}}+2z\tau\right)\left(\frac{\partial}{\partial z}-2\overline{z} \tau\right)\right)
	\end{align*}
	\begin{align} \label{24}
		=-\Delta+4\left(y^{2}_{1}+y^{2}_{2}\right)\tau^{2}+4i\left(y_{1}\frac{\partial}{\partial y_{2}}-y_{2}\frac{\partial}{\partial y_{1}}\right)\tau
	\end{align}
	\begin{rmk}
		The fundamental connection between the \textit{sub-laplacian} and the \textit{twisted laplacian} is given by the following result. ( \textbf{ref.} \cite{20}, \cite{21}, \cite{22} ) \end{rmk}
	\begin{thm}
		Suppose, $u\in \mathcal{S}'(\mathcal{H}_{3})\cap C^{\infty}(\mathcal{H}_{3})$ be such that, $\check{u}(z,\tau)$ is a tempered distribution of $\tau$ on $\mathbb{R}$,  $\forall$ \hspace{10pt} $z\in \mathbb{C}$, where, $\check{u}$ denotes the \textit{Inverse Fourier Transform} of $u$ with respect to time $t$. Then, for almost every $\tau\in \mathbb{R}\setminus \{0\}$,
		\begin{align*}
			\left(\mathcal{L}u\right)^{\tau}=L_{\tau}u^{\tau}
		\end{align*}
		where,
		\begin{align*}
			\left(\mathcal{L}u\right)^{\tau}(z)=\left(\mathcal{L}u\right)^{\check{ }}(z)\mbox{ , \hspace{20pt}}z\in \mathbb{C}
		\end{align*}
		and,
		\begin{align*}
			u^{\tau}(z)=\check{u}(z,\tau)\mbox{ , \hspace{20pt}}z\in \mathbb{C}
		\end{align*}
	\end{thm}
	We recall the definition of \textbf{Fourier Transform} $\hat{f}$ of a function $f\in L^{1}(\mathbb{R})$ as,
	\begin{align} \label{25}
		\hat{f}(\xi):=\frac{1}{\sqrt{2\pi}}\int\limits_{-\infty}^{\infty}e^{-ix.\xi}f(x)dx\mbox{ , \hspace{20pt}}\xi \in \mathbb{R}.
	\end{align}
	Since, our primary intention is to study the \textit{spectral properties }of $L_{\tau}$, we introduce the following.
	\begin{dfn} \label{dfn17}
		The \textbf{Fourier-Wigner Transform} $V_{\tau}(f,g)$ of the functions $f,g\in \mathcal{S}(\mathbb{R})$ is defined by,
		\begin{align*}
			V_{\tau}(f,g)(q,p)=\frac{1}{\sqrt{2\pi}}|\tau|^{1/2}\int\limits_{-\infty}^{\infty}e^{i\tau q.y}f(y-2p)\overline{g(y+2p)}dy \mbox{ , \hspace{20pt}}\forall \hspace{10pt} q,p\in \mathbb{R}
		\end{align*}
	\end{dfn}
	If, $\tau=1$, then, $V_{1}(f,g)=V(f,g)$ , which in fact, defines the \textit{Classical Fourier-Wigner Transform}. ( \textbf{ref.} \cite{25},\cite{26}, \cite{27} ) \par 
	It can be further established that, 
	\begin{align*}
		V_{\tau}(f,g)(q,p)=|\tau|^{1/2}V(f,g)(\tau q,p)
	\end{align*}
	For $\tau \in \mathbb{R}\setminus \{0\}$ and, $k=0,1,2, \cdots $, we define the function $e_{k,\tau}$ on $\mathbb{R}$ by,
	\begin{align*}
		e_{k,\tau}(x)=|\tau|^{1/4}e_{k}(\sqrt{|\tau|}x)\mbox{ , \hspace{20pt}}x\in \mathbb{R}
	\end{align*}
	Where, $e_{k}$ denotes the \textbf{Hermite Function} defined as follows,
	\begin{align} \label{26}
		e_{k}(x):=\frac{1}{\left(2^{k}k!\sqrt{\pi}\right)^{1/2}}e^{-\frac{x^{2}}{2}}H_{k}(x) \mbox{ , \hspace{20pt}} x\in \mathbb{R}.
	\end{align}
	Such that,
	\begin{align} \label{27}
		H_{k}(x):=(-1)^{k}e^{x^{2}}\left(\frac{d}{dx}\right)^{k}\left(e^{-x^{2}}\right)\mbox{ , \hspace{20pt}}x\in \mathbb{R}.
	\end{align}
	Given $j,k=0,1,2,...$, we define the function $e_{j,k,\tau}$ on $\mathbb{C}$ as follows,
	\begin{align} \label{28}
		e_{j,k,\tau}=V_{\tau}(e_{j,\tau}\mbox{ , } e_{k,\tau})
	\end{align}
	Further computation yields,
	\begin{align*}
		e_{j,k,1}=V_{1}(e_{j,1}\mbox{ , } e_{k,1})=V(e_{j}\mbox{ , }e_{k})
	\end{align*}
	Where, $V(e_{j}\mbox{ , }e_{k})$ is the \textit{Classical Hermite Function} on $\mathbb{C}$. [ \textbf{ref.} \cite{25} ]
	\begin{rmk}
		The above result can also be interpreted as an analogue of \cite[Proposition~21.1]{25}.
	\end{rmk}
	\begin{prop} \label{prop3}
		The set, $\{e_{j,k,\tau}\mbox{ }|\mbox{ }j,k=0,1,2,....\}$ is an \textit{orthonormal basis} for $L^{2}(\mathbb{C})$.
	\end{prop}
	Indeed we can learn about the respective \textit{spectral properties} of $L_{\tau} \mbox{  ,  }\tau\in \mathbb{R}\setminus \{0\}$.
	\begin{thm}\label{thm3}
		For $j,k=0,1,2,....$, the following holds true,
		\begin{align*}
			L_{\tau}e_{j,k,\tau}=(2k+1)|\tau|e_{j,k,\tau}
		\end{align*}
	\end{thm}
	\begin{proof}
		Proof is similar to the derivation cited in \cite[Theorem~22.2]{25}.\par 
		A priori from the statement of \cite[Theorem~22.1]{25}, it follows that, for $j=0,1,2,...$ and $k=0,1,2,.....$;
		\begin{align*}
			Z\overline{Z}e_{j,k}=i(2k+2)^{1/2}Ze_{j,k-1}=-(2k+2)e_{j,k}
		\end{align*}
		and,
		\begin{align*}
			\overline{Z}Ze_{j,k}=i(2k)^{1/2}Ze_{j,k-1}=-(2k)e_{j,k}
		\end{align*}
		Thus, \begin{align*}
			L_{\tau}e_{j,k}=-\frac{1}{2}\left(Z\overline{Z}+\overline{Z}Z\right)|\tau|e_{j,k,\tau}
		\end{align*}
		Important to observe that, the above identity also holds for $e_{j,0,\tau}$, where, $j=0,1,2,.....$ due to the fact that,
		\begin{align*}
			Ze_{j,0}=0\mbox{ , \hspace{20pt}}\forall \hspace{10pt} j=0,1,2,....
		\end{align*}
	\end{proof}
	\subsection{Essential Self-Adjointness Property}
	Our aim in this section is to study the \textit{Sub-Laplacian} $\mathcal{L}$ as an \textit{unbounded linear operator} from $L^{2}(\mathcal{H}_{3})$ to  $L^{2}(\mathcal{H}_{3})$ with dense domain denoted as, $\mathcal{S}(\mathcal{H}_{3})$.
	\begin{prop}
		$\mathcal{L}$ is an \textit{injective symmetric} operator from $L^{2}(\mathcal{H}_{3})$ to  $L^{2}(\mathcal{H}_{3})$ with dense domain $\mathcal{S}(\mathcal{H}_{3})$. Furthermore, it is strictly positive.
	\end{prop}
	Proof follows from integration by parts.
	\begin{rmk}
		The above proposition implies that, $\mathcal{L}$ is closed. Suppose we denote its \textit{closure} by $\mathcal{L}_{0}$. Hence, $\mathcal{L}_{0}$ is closed, symmetric and positive operator from $L^{2}(\mathcal{H}_{3})$ to  $L^{2}(\mathcal{H}_{3})$.
	\end{rmk}
	In fact, $\mathcal{L}$ is \textbf{Essentially Self-Adjoint} ( \textbf{ref.} \cite[Section~4, pp.~1603]{23}) in the following sense that, it has a unique \textit{self-adjoint extension}, which, subsequently equals to $\mathcal{L}_{0}$.
	Further details on \textit{essential self-adjointness} can be found in \cite[Theorem~X.23]{24}.
	\begin{rmk}
		The results and derivations in this article are also valid for the \textit{Sub-Laplacian} on the $n$-dimensional \textit{Heisenberg Group} $\mathcal{H}_{n}$, $n>1$, having an underlying space as $\mathbb{C}^{n}\times \mathbb{R}$, although, we have only explored the case for $n=1$, which is $\mathcal{H}_{3}$ for the sake of lucidity.
	\end{rmk}
	\section{The Spectrum of the sub-Laplacian}
	A priori given  a \textit{closed linear operator} $\mathcal{T}$ from a complex Banach Space $X$ with dense domain $\mathcal{D}(\mathcal{T})$, we provide the following definitions,
	\begin{dfn}\label{dfn18}
		(Spectrum)  The \textbf{Resolvent Set} $\rho(\mathcal{T})$ of $\mathcal{T}$ is defined as follows,
		\begin{align*}
			\rho(\mathcal{T})\mbox{ }:=\mbox{ }\left\{\lambda\in \mathbb{C}\mbox{ }|\mbox{ }\mathcal{T}-\lambda I\mbox{ }:\mbox{ }\mathcal{D}(\mathcal{T})\longrightarrow X \mbox{  is bijective}\right\}
		\end{align*}
		Where, $I$ denotes the \textbf{identity operator} on $X$. \\
		The \textbf{Spectrum}, denoted by $\Sigma(\mathcal{T})$ is defined to be the complement of $\rho(\mathcal{T})$ in $\mathbb{C}$.
	\end{dfn}
	\begin{dfn}\label{dfn19}
		(Point Spectrum)  The \textbf{point spectrum} \cite{28} of $\mathcal{T}$, denoted by $\Sigma_{p}(\mathcal{T})$ is defined as,
		\begin{align*}
			\Sigma_{p}(\mathcal{T})\mbox{ }:=\mbox{ }\left\{\lambda\in \mathbb{C}\mbox{ }|\mbox{ }\mathcal{T}-\lambda I\mbox{ }:\mbox{ }\mathcal{D}(\mathcal{T})\longrightarrow X \mbox{  is not injective}\right\}
		\end{align*}
	\end{dfn}
	\begin{dfn}\label{dfn20}
		(Continuous Spectrum)  The \textbf{Continuous Spectrum} of $\mathcal{T}$, denoted by $\Sigma_{c}(\mathcal{T})$, is defined as,
		\begin{align*}
			\Sigma_{c}(\mathcal{T})\mbox{ }:=\mbox{ }\left\{\lambda\in \mathbb{C}\mbox{ }|\mbox{ }Range(\mathcal{T}-\lambda I)\mbox{ } \mbox{  is dense in }X\mbox{ ,  }(\mathcal{T}-\lambda I)^{-1}\mbox{ exists, but is unbounded}\right\}
		\end{align*}
	\end{dfn}
	\begin{dfn}\label{dfn21}
		(Residual Spectrum)  The \textbf{Residual Spectrum} of $\mathcal{T}$, denoted by $\Sigma_{r}(\mathcal{T})$, is defined as,
		\begin{align*}
			\Sigma_{r}(\mathcal{T})\mbox{ }:=\mbox{ }\left\{\lambda\in \mathbb{C}\mbox{ }|\mbox{ }Range(\mathcal{T}-\lambda I)\mbox{ } \mbox{  is not dense in }X\mbox{ ,  }(\mathcal{T}-\lambda I)^{-1}\mbox{ exists and is bounded}\right\}
		\end{align*}
	\end{dfn}
	We can indeed deduce that, $\Sigma_{p}(\mathcal{T})$, $\Sigma_{c}(\mathcal{T})$ and $\Sigma_{r}(\mathcal{T})$ are \textit{mutually disjoint}. Furthermore,
	\begin{align*}
		\Sigma(\mathcal{T})=\Sigma_{p}(\mathcal{T})+ \Sigma_{c}(\mathcal{T})+ \Sigma_{r}(\mathcal{T})
	\end{align*}
	\begin{prop}\label{prop4}
		A priori given a complex and separable Hilbert Space $X$, if $\mathcal{T}$ is a self-adjoint operator , then,
		\begin{align*}
			\Sigma_{r}(\mathcal{T})=\phi
		\end{align*}
	\end{prop}
	With the above notations and concepts, we thus delineate a more precise illustration of the Spectrum of the \textit{sub-Laplacian} on the \textit{Heisenberg Group}.
	\begin{thm}\label{thm4}
		We have,
		\begin{align*}
			\Sigma(\mathcal{L}_{0})=\Sigma_{c}(\mathcal{L}_{0})=\left[0,\infty\right)
		\end{align*}
	\end{thm}
	\begin{proof}
		We first intend to show that, no eigenvalue of $\mathcal{L}_{0}$ lies in the interval $\left[0,\infty \right)$. We know for a fact that, $0$ is not an eigenvalue of $\mathcal{L}_{0}$ ( Proof mentioned in \cite{22}). Suppose, $\lambda$ be a positive number such that, $\exists$ a function $u\in L^{2}(\mathcal{H}_{3})$ satisfying,
		\begin{align*}
			\mathcal{L}_{0}u=\lambda u
		\end{align*}
		Consequently,
		\begin{align*}
			L_{\tau}u^{\tau}=\lambda u^{\tau}
		\end{align*}
		Although, the above relation implies that, $u^{\tau}=0$, \hspace{10pt}$\forall$ $\tau\in \mathbb{R}\setminus \{0\}$ and,
		\begin{align}
			|\tau|\neq \frac{\lambda}{(2k+1)}\mbox{ , }\hspace{20pt}k=0,1,2,\cdots 
		\end{align}
		This helps us conclude that, $u=0$, a contradiction. Moreover, $\mathcal{L}_{0}$ being \textit{self-adjoint}, it implies,
		\begin{align*}
			\Sigma(\mathcal{L}_{0})=\Sigma_{c}(\mathcal{L}_{0})
		\end{align*}
		Thus, it only suffices to establish that, $(\mathcal{L}_{0}-\lambda I)$ is not \textit{surjective}  $\forall$ \hspace{10pt}$\lambda\in \left[0,\infty\right)$. \par 
		Assuming $(\mathcal{L}_{0}-\lambda I)$ to be surjective for some $\lambda_{0}\in \left[0,\infty\right)$, we can infer that, $\lambda_{0}\in \rho(\mathcal{L}_{0})$. Hence, $\exists$ an open interval $I_{\lambda_{0}}\subset \rho(\mathcal{L}_{0})$ containing $\lambda_{0}$.\par 
		Define $f$ on $\mathcal{H}$ as,
		\begin{align*}
			f(x,y,t)=h(x,y)e^{-\frac{t^{2}}{2}}\mbox{ , }\hspace{10pt} x,y,t\in \mathbb{R}
		\end{align*}
		Where, $h\in L^{2}(\mathbb{R}^{2})$.\par 
		Therefore, for every $\lambda \in I_{\lambda_{0}}$, $\exists$ a function $u_{\lambda}\in L^{2}(\mathcal{H}_{3})$, such that,
		\begin{align*}
			(\mathcal{L}_{0}-\lambda I)u_{\lambda}=f
		\end{align*}
		Computing the \textit{Inverse Transform} with respect to $t$ helps us conclude,
		\begin{align}
			(L_{\tau}-\lambda I)u^{\tau}_{\lambda}=he^{-\frac{\tau^{2}}{2}}\mbox{ , }\hspace{20pt} \mbox{ for almost every }\tau\in \mathbb{R}\setminus\{0\}
		\end{align}
		As a consequence, $(L_{\tau}-\lambda I)$ is \textit{surjective} $\forall$\hspace{10pt}$\tau\in S_{\lambda}$ for which the \textit{lebesgue measure}, 
		\begin{align*}
			m(\mathbb{R}\setminus S_{\lambda})=0
		\end{align*}
		Consider, $\tau\in \bigcap\limits_{r\in I_{\lambda_{0}}\cap \mathbb{Q}}S_{r}$ , $\mathbb{Q}$ being the set of all \textit{rationals}. Hence, $(L_{\tau}-\lambda I)$ is surjective, and subsequently \textit{injective} for every $\lambda\in I_{\lambda_{0}}\cap \mathbb{Q}$.\par 
		Thus, $L_{\tau}-\lambda I$ is \textit{bijective}, \hspace{5pt} $\forall$ $\lambda\in I_{\lambda_{0}}$. ( Use the fact that, the \textit{Resolvent Set} of $L_{\tau}$ is \textit{open} ) Furthermore, we can also observe that, $(L_{\tau}-\lambda I)$ is \textit{injective} iff, 
		\begin{align*}
			\lambda\neq (2k+1)|\tau|\mbox{ , }\hspace{20pt}k=0,1,2,\cdots 
		\end{align*}
		A contradiction to the fact that, if we assume $\tau\in \bigcap\limits_{r\in I_{\lambda_{0}}}S_{r}$ to be sufficiently small, such that, $(2k+1)|\tau|\in I_{\lambda_{0}}$ for some $k=0,1,2,\cdots$. Hence, the proof is done.
	\end{proof}
	\begin{rmk}\label{rmk5}
		As an application to \textit{Theorem }\eqref{thm4}, in the next section, we shall introoduce various \textit{Essential Spectra} which'll be extremely helpful to us.
	\end{rmk}
	\section{Essential Spectra of sub-Laplacian}
	Given a \textit{closed} linear operator $\mathcal{T}$ densly defined on a complex Banach Space $X$. 
	\begin{dfn}\label{dfn22}
		The \textbf{Essential Spectrum} of $\mathcal{T}$, denoted as $\Sigma_{DS}(\mathcal{T})$  [\textit{Dunford} and Schwartz \cite{29} ] is defined as follows,
		\begin{align*}
			\Sigma_{DS}(\mathcal{T})\mbox{ }:=\mbox{ }\left\{\lambda\in \mathbb{C}\mbox{ }|\mbox{ }Range(\mathcal{T}-\lambda I)\mbox{ } \mbox{  is not closed in }X\right\}
		\end{align*}
	\end{dfn}
	Let us recall the following concept from \textit{Functional Analysis}.
	\begin{dfn}
		(Fredholm Operator)  Given any two Banach Spaces $X$ and $Y$, and a bounded linear operator $\mathcal{T}:X\longrightarrow Y$, we define $\mathcal{T}$ to be \textbf{Fredholm} if, the following conditions hold true:
		\begin{enumerate}
			\item $ker(\mathcal{T})$ is of finite dimension.
			\item $Range(\mathcal{T})$ is closed.
			\item $Coker(\mathcal{T})$ is of finite dimension.
		\end{enumerate}
		If $\mathcal{T}$ is \textit{Fredholm}, then, \textbf{Index of }$\mathcal{T}$ is defined to be equal to $\{dim(ker(\mathcal{T}))-dim(Coker(\mathcal{T}))\}$.
	\end{dfn}
	Denote $\Phi_{W}(\mathcal{T})$ to be the set of all $\lambda \in \mathbb{C}$, such that, $\mathcal{T}-\lambda I$ is a \textit{Fredholm Operator}.
	Furthermore, suppose, $\Phi_{S}(\mathcal{T})$ be the set of all complex numbers $\lambda$ satisfying, $\mathcal{T}-\lambda I$ is \textit{Fredholm} with index $0$. \par 
	Then, the essential spectrums $\Sigma_{W}(\mathcal{T})$ [Wolf \cite{30} \cite{31}] and, $\Sigma_{S}(\mathcal{T})$ [Schechter \cite{32}] of $\mathcal{T}$ having the following definitions,
	\begin{align*}
		\Sigma_{W}(\mathcal{T})=\mathbb{C}\setminus \Phi_{W}(\mathcal{T})
	\end{align*}
	and,
	\begin{align*}
		\Sigma_{S}(\mathcal{T})=\mathbb{C}\setminus \Phi_{S}(\mathcal{T})
	\end{align*}
	It is obvious from the respective definitions that,
	\begin{align}
		\Sigma_{DS}(\mathcal{T})\subseteq \Sigma_{W}(\mathcal{T})\subseteq \Sigma_{S}(\mathcal{T})
	\end{align}
	In the particular case when, $\mathcal{T}$ denotes the \textit{sub-laplacian} on the \textit{Heisenberg Group} $\mathcal{H}_{3}$, we can indeed deduce the following important result.
	\begin{thm}\label{thm5}
		We shall have,
		\begin{align}
			\Sigma_{DS}(\mathcal{L}_{0})=\Sigma_{W}(\mathcal{L}_{0})=\Sigma_{S}(\mathcal{L}_{0})=\left[0,\infty\right)
		\end{align}
	\end{thm}
	\begin{proof}
		It only requires us to verify that,
		\begin{align*}
			\left[0,\infty\right)\subseteq \Sigma_{DS}(\mathcal{L}_{0})
		\end{align*}
		Assume, $\lambda\in \left[0,\infty\right)$, although, $\lambda\notin \Sigma_{DS}(\mathcal{L}_{0})$. Thus, $Range(\mathcal{L}_{0}-\lambda I)$ is \textit{closed} in $L^{2}(\mathcal{H}_{3})$. This in turn helps us conclude that, $\mathcal{L}_{0}-\lambda I$ is \textit{bijective}, i.e., $\lambda\in \mathcal{L}_{0}$, a contradiction. Hence the proof is complete.
	\end{proof}
	\begin{rmk}\label{rmk6}
		Similar technique can be implemented to compute the \textit{Spectrum} of the unique \textbf{self-adjoint extension} $\Delta_{\mathcal{H}_{3},0}$ of the \textit{Laplacian} $\Delta_{\mathcal{H}_{3}}$ on the \textit{Heisenberg Group} $\mathcal{H}_{3}$ defined as,
		\begin{align*}
			\Delta_{\mathcal{H}_{3}}=-\left(X^{2}+Y^{2}+T^{2}\right)
		\end{align*}
		In fact, we have \cite{33},
		\begin{align*}
			\Sigma\left(\Delta_{\mathcal{H}_{3},0}\right)=\Sigma_{c}\left(\Delta_{\mathcal{H}_{3},0}\right)=\left[0,\infty\right)
		\end{align*}
		Therefore, we can infer,
		\begin{align}
			\Sigma_{DS}(\Delta_{\mathcal{H}_{3},0})=\Sigma_{W}(\Delta_{\mathcal{H}_{3},0})=\Sigma_{S}(\Delta_{\mathcal{H}_{3},0})=\left[0,\infty\right)
		\end{align}
	\end{rmk}
	\section{Kirchhoff-type Critical Elliptic Equations involving $p$-sub-Laplacians on $\mathcal{H}_{n}$ }
	\subsection{Some Important Concepts}
	A priori given a generalized \textit{Heisenberg Group} $\mathcal{H}_{n}$, a \textit{lie group} of topological dimension $(2n+1)$, having $\mathbb{R}^{2n+1}$ as a background manifold, endowd with the non-Abelian group law,
	\begin{align*}
		\tau\mbox{ }:\mbox{ }\mathcal{H}_{n}\longrightarrow \mathcal
		{H}
		_{n} \mbox{ , }\hspace{10pt} \tau_{\xi}(\xi')=\xi\circ \xi'
	\end{align*}
	where,
	\begin{align*}
		\xi \circ \xi'=\left(x+x',y+y',t+t'+2\sum\limits_{i=1}^{n}(y_{i}x_{i}'-x_{i}y_{i}')\right) \mbox{ , }\hspace{10pt} \forall \hspace{10pt} \xi,\xi' \in \mathcal{H}_{n}
	\end{align*}
	Subsequently, inverse of this operation can be deduced as, $\xi ^{-1}=-\xi$, thus, 
	\begin{align*}
		(\xi\circ \xi')^{-1}=(\xi')^{-1}\circ (\xi)^{-1}
	\end{align*}
	Applying similar concepts as in \eqref{16} the corresponding Lie Algebra of \textit{left-invariant vector fields} is generated by,
	\begin{align}\label{29}
		X_{j}=\partial_{x_{j}}+2y_{j}\partial_{t}\mbox{ , \hspace{20pt}}Y_{j}=\partial_{y_{j}}-2x_{j}\partial_{t}\mbox{ , \hspace{20pt}}T=4\partial_{t}
	\end{align}
	For every $j=1,2,3,\cdots, n $.
	As a consequence, the basis $\beta=\{X_{j},Y_{j},T\}_{j=1(1)n}$ satisfies the \textit{Heisenberg Canonical Communication Relations} for position and momentum,
	\begin{align*}
		[X_{j},Y_{j}]=-\delta_{jk}T
	\end{align*}
	And all other commutators are \textit{zero}.
	\begin{rmk}\label{rmk7}
		A vector field in the span of $\beta$ is called \textbf{Horizontal}.
	\end{rmk}
	\begin{dfn}
		(Kor$\acute{a}$nyi Norm)  It can be observed that, the anisotropic dilation structure on the Heisenberg Group $\mathcal{H}_{n}$ induces the \textbf{Kor$\acute{a}$nyi Norm} defined as follows: 
		\begin{align*}
			r(\xi):=r(z,t)=(|z|^{4}+t^{2})^{\frac{1}{4}} \mbox{ , }\hspace{10pt}\forall \hspace{5pt} \xi=(z,t)\in \mathcal{H}_{n}
		\end{align*} 
	\end{dfn}
	Some properties of Kor$\acute{a}$nyi Norm include that, its \textit{homogeneous degree} with resopect to dilations is equal to $1$.\par 
	Subsequently, the \textbf{Kor$\acute{a}$nyi Distance} is defined as :
	\begin{align*}
		d_{H}(\xi,\xi')=r(\xi^{-1}\circ \xi')\mbox{ , }\hspace{10pt}\forall \hspace{5pt} (\xi,\xi')\in \mathcal{H}_{n}\times \mathcal{H}_{n}
	\end{align*}
	And, the \textbf{Kor$\acute{a}$nyi Open Ball} of radius $R$ centered at $\xi_{0}$ is, 
	\begin{align*}
		B_{R}(\xi_{0})=\{\xi\in \mathcal{H}_{n}\mbox{ }|\mbox{ }d_{H}(\xi,\xi_{0})<R\}
	\end{align*}
	\begin{rmk}\label{rmk8}
		We can indeed infer that, the \textit{Haar Measure} on $\mathcal{H}_{n}$ is consistent with the \textit{Lebesgue Measure} on $\mathbb{R}^{2n+1}$, and is invariant under the left translations of $\mathcal{H}_{n}$. Moreover, it is $Q$-Homogeneous with respect to dilations ($Q$ denotes the \textit{Hausdorff Dimension}).\par 
		Thus, the \textit{topological dimension} of $\mathcal{H}_{n}$ (is equal to $2n+1$) is strictly less than $Q=2n+2$.
	\end{rmk}
	\begin{dfn}
		We define the \textbf{Horizontal Gradient} of a $C^{1}$-function $u:\mathcal{H}_{n}\longrightarrow \mathbb{R}$ as:
		\begin{align}
			D_{H}(u)=\sum\limits_{j=1}^{n}\left((X_{j}u)X_{j}+(Y_{j}u)Y_{j}\right)
		\end{align}
	\end{dfn}
	An important observation is that, $D_{H}u$ is in fact an element of $span(\beta)$. Thus, we can define the natural inner product in $span(\beta)$ as :
	\begin{align*}
		(X,Y)_{H}:=\sum\limits_{j=1}^{n}\left(x^{j}y^{j}+\tilde{x}^{j}\tilde{y}^{j}\right)
	\end{align*}
	For every $X=\{x^{j}X_{j}+\tilde{x}^{j}Y_{j}\}_{j=1(1)n}$ , $Y=\{y^{j}X_{j}+\tilde{y}^{j}Y_{j}\}_{j=1(1)n}$. This eventually helps us define the \textbf{Hilbertian Norm},
	\begin{align}
		|D_{H}(u)|:=\sqrt{\left(D_{H}(u),D_{H}(u)\right)_{H} }
	\end{align}
	for any horizontal vector field $D_{H}(u)$.
	\begin{dfn}\label{dfn23}
		Given any horizontal vector field function, $X=X(\xi)$, $X=\{x^{j}X_{j}+\tilde{x}^{j}Y_{j}\}_{j=1(1)n}$ of class $C^{1}(\mathcal{H}_{n},\mathbb{R}^{2n})$, the \textbf{Horizontal Divergence} of $X$ is defined as,
		\begin{align*}
			div_{H}X=\sum\limits_{j=1}^{n}\left(X_{j}(x^{j})+Y_j(y^{j})\right)
		\end{align*}
	\end{dfn}
	We can generalize the notion of \textit{sub-Laplacians} in the case for $\mathcal{H}_{3}$ to a generalized Heisenberg Group $\mathcal{H}_{n}$.
	\begin{dfn}
		(sub-Laplacian)  For every $u\in C^{2}(\mathcal{H}_{n})$, the \textbf{sub-Laplacian} or, \textit{Kohn-Spencer Laplacian} of $u$ is defined as:
		\begin{align*}
			\Delta_{H}(u)=\sum\limits_{j=1}^{n}\left(X^{2}_{j}+Y^{2}_{j}\right)u
		\end{align*} 
		\begin{align}
			=\sum\limits_{j=1}^{n}\left(\frac{\partial^{2}}{\partial x^{2}_{j}}+\frac{\partial^{2}}{\partial y^{2}_{j}}+4y_{j}\frac{\partial^{2}}{\partial x_{j}\partial t}-4x_{j}\frac{\partial^{2}}{\partial y_{j}\partial t}\right)u+4|z|^{2}\frac{\partial^{2}u}{\partial^{2} t}
		\end{align}
	\end{dfn}
	H$\ddot{o}$rmander \cite{18} established the fact that, $\Delta_{H}$ is \textit{Hypoelliptic}. To be more precise,
	\begin{align*}
		\Delta_{H}(u)=div_{H}D_{H}(u) \mbox{ , }\hspace{10pt} \forall \hspace{10pt} u\in C^{2}(\mathcal{H}_{n})
	\end{align*}
	We can in fact further generalize the \textit{Kohn-Spencer Laplacian} to obtain the so called \textbf{p-Laplacian} on $\mathcal{H}_{n}$, having the following expression:
	\begin{align}
		\Delta_{H,p}(\varphi)=div_{H}\left( |D_{H}(\varphi) |^{p-2}_{H}D_{H}(\varphi)\right)
	\end{align}
	for every $\varphi \in C^{\infty}_{c}(\mathcal{H}_{n})$.For further study, interested readers can refer to \cite{34}, \cite{35}, \cite{36}, \cite{37}.\par 
	let us recall some significant properties of \textit{classic Sobolev Spaces} on $\mathcal{H}_{n}$. 
	\begin{dfn}
		The standard \textbf{$L^{p}$-norm} is defined as,
		\begin{align*}
			||u||^{p}_{p}=\int\limits_{\omega}|u|^{p}d\xi \mbox{ , }\hspace{10pt} \forall \hspace{5pt} u\in \Omega
		\end{align*}
	\end{dfn}
	A priori given $\Omega$ to be a \textit{bounded Lipschitz Domain} in $\mathcal{H}_{n}$ or, $\Omega=\mathcal{H}_{n}$. Then we denote $W^{1,p}(\Omega)$ as the \textit{Horizontal Sobolev Space} of the functions $u\in L^{p}(\Omega)$, provided, $D_{H}(u)$ exists in the sense of \textit{distributions}, and furthermore, $|D_{H}(u)|\in L^{p}(\Omega)$, endowed with the norm,
	\begin{align*}
		||u||_{W^{1,p}(\Omega)}=\left(||u||^{p}_{p}+||D_{H}(u)||^{p}_{p}\right)^{\frac{1}{p}}
	\end{align*}
	Consider the function space,
	\begin{align*}
		HW^{1,p}_{V}(\mathcal{H}_{n}):= \left\{u\in W^{1,p}(\mathcal{H}_{n}):\int\limits_{\mathcal{H}_{n}}V(\xi)|u(\xi)|^{p}d\xi<\infty\right\}
	\end{align*}
	with the followig norm defined on it,
	\begin{align}
		||u||=||u||_{HW^{1,p}_{V}(\mathcal{H}_{n})}:=\left(||D_{H}(u)||^{p}_{p}+||u||^{p}_{p,V}\right)^{\frac{1}{p}}
	\end{align}
	and,
	\begin{align}
		||u||^{p}_{p,V}=\int\limits_{\mathcal{H}_{n}}V(\xi)|u(\xi)|^{p}d\xi
	\end{align}
	Where, $V$ denotes the \textbf{potenial function}. Furthermore, under the assumption that, $V(\xi)\geq V_{0}>0$, we can in fact conclude that, $	HW^{1,p}_{V}(\mathcal{H}_{n})$ is a  \textit{reflexive Banach Space}. For proof involving Euclidean setting, it can be found in \cite[Lemma~10]{38}, whereas, in case for $\mathcal{H}_{n}$, we shall be needing few minor alterations. The continuous embedding of $HW^{1,p}_{V}(\mathcal{H}_{n})\hookrightarrow W^{1,p}(\mathcal{H}_{n})\hookrightarrow L^{t}(\mathcal{H}_{n})$\hspace{10pt} $\forall$ \hspace{5pt} $p\leq t<p^{*}$, where, $p^{*}:=\frac{Qp}{Q-p}$ is the \textit{Critical Sobolev Exponent} on $\mathcal{H}_{n}$.
	\begin{rmk}
		In fact, one can establish that, the best value of the constant $V_{0}$, denoted by $C_{p^{*}}$ is attained in the \textit{Folland-Stein Spaces} $S^{1,p}(\mathcal{H}_{n})$, which, also can be interpreted as the completion of $C^{\infty}_{c}(\mathcal{H}_{n})$ in terms of the norm,
		\begin{align*}
			||D_{H}(u)||_{p}=\left(\hspace{5pt} \int\limits_{\mathcal{H}_{n}}|D_{H}(u)|^{p}_{H}d\xi \hspace{5pt}\right)^{\frac{1}{p}}
		\end{align*}
		Thus, we can obtain the following estimate of $C_{p^{*}}$ of the \textit{Folland-Stein Inequality} as,
		\begin{align*}
			C_{p^{*}}=\inf\limits_{u\in S^{1,p}(\mathcal{H}_{n}),u\neq 0}\frac{||D_{H}(u)||^{p}_{p}}{||u||^{p}_{p^{*}}}
		\end{align*}
		For further details, see \cite{39}.
	\end{rmk}
	\subsection{Introduction to Critical Kirchhoff Equations}
	In this section, we shall deal with a class of \textit{Kirchhoff}-type Critical \textit{Elliptic Equations} (Kirchhoff, $1883$) as a generalization of \textit{D'Alembert}'s Wave Equation for free vibrations of elastic strings, involving $p$-sub-Laplacians, having the following representation,
	\begin{align}\label{30}
		M\left(||D_{H}(u)||^{p}_{p}+||u||^{p}_{p,V}\right)\left\{-\Delta_{H,p}(u)+V(\xi)|u|^{p-2}u\right\}=\lambda f(\xi,u)+|u|^{p^{*}-2}u 
	\end{align}
	\begin{align*}
		\hspace{100pt}	\xi \in \mathcal{H}_{n} \mbox{ , }\hspace{20pt}u\in HW^{1,p}_{V}(\mathcal{H}_{n})
	\end{align*}
	in both non-degenerate and degenerate cases separately. 
	\begin{rmk}
		In \eqref{30}, $\lambda$ is a real parameter, and, $M$ denotes the \textit{Kirchhoff Function}.
	\end{rmk}
	A priori, for pre-determined constants $\rho, P_{0}, h, E, L$ having some physical interpretation, Kirchhoff thus established a model given by the equation,
	\begin{align*}
		\rho \frac{\partial^{2}u}{\partial t^{2}}-\left(\frac{P_{0}}{h}+\frac{E}{2L}\int\limits_{0}^{L} \left|\frac{\partial u}{\partial x}\right|^{2}dx\right)\frac{\partial^{2}u}{\partial x^{2}}=0
	\end{align*}
	In particular, the study of critical Kirchhoff-type problems were first initially studied in the seminal paper of \textit{Br$\acute{e}$zis} \& \textit{Nirenberg} [ref. \cite{40}], in which their intention was to study the Laplacian equations.\par 
	Over the years, there have been many generalizations of \cite{40} in various directions. For instance, \textit{Liao et al.} \cite{41} studied the following non-local problem with Critical Sobolev Exponent of the form,\\
	\begin{align}\label{31}
		- \left\{  a+b\int\limits_{\Omega}\left| \nabla u\right|^{2}dx  \right\}\Delta u= \mu |u|^{2^{*}-2}u+\lambda |u|^{q-2}u \mbox{ ,  }\hspace{20pt} x\in \Omega 
	\end{align}
	\begin{align*}
		u=0 \mbox{ ,  }\hspace{10pt} x\in \partial \Omega
	\end{align*}
	where, $\Omega \subseteq \mathbb{R}^{N}$ ($N\geq 4$) is a smooth bounded domain, $2^{*}=\frac{2N}{N-2}$ is the \textit{Critical Sobolev Exponent}. The existence an multiplicity of the solutions of \eqref{31} are obtained by applying the Variational Methods and the \textit{Critical Point Theorem}. \textit{Liang et al.} \cite{42} effectively followed the similar approach to derive the solutions to the fractional Schr$\ddot{o}$dinger-Kirchhoff equations with electro-magnetic fields and critical non-linearity in the no-degenerate Kirchhoff case by using the fractional versions of the \textit{concentration compactness principle} and\textit{variational methods}. Results related to the existence of solution in case of non-degenerate Kirchhoff problems are illustrated, for example, in \cite{43}, \cite{44}, \cite{45} and \cite{38}.\par 
	Furthermore, there are extensive amount of research which are currently going on for the \textit{degenerate }case. Interested readers can look at the findings of \textit{wang et al.} \cite{46}  and even \textit{Song \& Shi} \cite{47} in this regard.\par 
	The motivation behind studying the problem \eqref{30}  primarily originates from the significant applications of the Heisenberg Group. \textit{Liang \& Pucci} \cite{48} considered a class of critical Kirchhoff-Poisson systems in the Heisenberg group under suitable assumptions. On the contrary, the existence of multiple solutions is obtained by using the symmetric \textbf{Mountain Pass Theorem}. A priori applying this result along with \textit{Singular Trudinger-Moser Inequality}, \textit{Deng \& Tian} \cite{49} discussed the existence of solutions for Kirchhoff-type systems involving $Q$-laplacian operator in the Heisenberg Group,
	\begin{align*}
		-K\left(\int\limits_{\Omega} \left|\nabla_{\mathcal{H}_{n}}u\right|^{Q} d\xi \right)\Delta_{Q}(u)=\lambda \frac{G_{u}(\xi,u,v)}{\rho(\xi)^{\wp}} \hspace{20pt} \mbox{in }\Omega
	\end{align*}
	\begin{align*}
		-K\left(\int\limits_{\Omega} \left|\nabla_{\mathcal{H}_{n}}v\right|^{Q} d\xi \right)\Delta_{Q}(v)=\lambda \frac{G_{v}(\xi,u,v)}{\rho(\xi)^{\wp}} \hspace{20pt} \mbox{in }\Omega
	\end{align*}
	\begin{align*}
		u=v=0 \hspace{20pt} \mbox{ on }\partial \Omega
	\end{align*}
	Where, $\Omega$ is an open, smooth and bounded subset of $\mathcal{H}_{n}$, $K$ is Kirchhoff-type function \& non-linear terms : $G_{u}$ and $G_{v}$ have critical exponent growth.\par 
	Whereas, \textit{Pucci \& Temeprini} \cite{50} studied the $(p,q)$ critical systems on the Heisenberg Group,
	\begin{align*}
		-div_{H}\left(A(\left|D_{H}(u)\right|_{H})\right)+B\left(|u|\right)u=\lambda H_{u}(u,v)+\frac{\alpha}{\wp ^{*}}|v|^{\beta}|u|^{\alpha-2}u
	\end{align*}
	\begin{align*}
		-div_{H}\left(A(\left|D_{H}(v)\right|_{H})\right)+B\left(|v|\right)v=\lambda H_{v}(u,v)+\frac{\beta}{\wp ^{*}}|u|^{\alpha}|v|^{\beta-2}v
	\end{align*}
	Subsequently, the existence of entire non-trivial solutions are obtained by applying the concentration-compactness principle in the vectorial Heisenberg context and variational methods.
	\begin{rmk}
		Further details on these results can be explored from \cite{51}.
	\end{rmk}
	\section{Existence and Multiplicity of Solutions}
	A priori having the concepts discussed in detail in the previous section, we now proceed towards proving the existence and multiplicity for a special class of Kirchhoff-type critical elliptic equations as mentioned in \eqref{30} involving the $p$-sub-Laplacian operators on $\mathcal{H}_{n}$ for both the \textit{degenerate} and the \textit{non-degenerate} case separately.\par
	\subsection{Some Important Assumptions}
	In order to establish our main results, a priori we shall assume that, $M:\mathbb{R}^{+}_{0}\longrightarrow \mathbb{R}^{+}_{0}$ is a continuous and non-decreasing function, the potential function $V$ and $M$. Therefore, they will satisfy the following properties :
	\begin{itemize}\label{32}
		\item $V:\mathcal{H}_{n}\longrightarrow \mathbb{R}_{+}$ is \textit{continuous}, and $\exists$  $V_{0}>0$ such that, $V\geq V_{0}>0$ in $\mathcal{H}_{n}$.
	\end{itemize}
	\begin{enumerate}\label{33}
		\item $\exists$ $m_{0}>0$ such that, $\inf\limits_{t\geq 0}M(t)=m_{0}$.
		\item $\exists$ $\tau\in \left[1,\frac{p^{*}}{p}\right)$ satisfying, 
		\begin{align*}
			\tau \mathcal{M}(t)\geq M(t)t \mbox{ , }\hspace{20pt} \forall \hspace{10pt}t\geq 0 
		\end{align*}
		Where, \begin{align*}
			\mathcal{M}(t):=\int\limits_{0}^{t}M(s)ds.
		\end{align*}
		\item $\exists$ $m_{1}>0$ such that, $M(t)\geq m_{1}t^{\tau-1}$ , \hspace{10pt} $\forall$ $t\geq 0$ and $M(0)=0$. 
	\end{enumerate}
	Furthermore, we impose the following hypotheses on the non-linearity of $f$:
	\begin{itemize}\label{34}
		\item $f:\mathcal{H}_{n}\times \mathbb{R}\longrightarrow \mathbb{R}$is a \textit{Carath$\acute{e}$odory Function} such that, $f$ is odd with respect to the second variable.
		\item $\exists$ a constant $r$ satisfying, $p^{*}>r>p\tau$ such that,
		\begin{align*}
			\left|f(\xi,t)\right|\leq a(\xi)\left|t\right|^{r-2}t \mbox{ , }\hspace{10pt} \mbox{ for a.e. }\xi\in \mathcal{H}_{n} \mbox{ and }t\in \mathbb{R}
		\end{align*} 
		where, $0\leq a(\xi)\in L^{\eta}(\mathcal{H}_{n})\cap L^{\infty}(\mathcal{H}_{n})$, and, $\eta:=\frac{p^{*}}{p^{*}-r}$, $\xi\in \mathcal{H}_{n}$.
		\item $\exists$ a constant $\theta$ satisfying $p\tau<\theta<p^{*}$ such that, $0<\theta F(\xi,t)\leq f(\xi,t)t$ , $\forall$  \hspace{10pt} $t\in \mathbb{R}_{+}$. We define, 
		\begin{align*}
			F(\xi,t):=\int\limits_{0}^{t}f(\xi,s)ds.
		\end{align*}
	\end{itemize}
	\subsection{Palais-Smale Condition $\left(PS\right)_{c}$}
	A priori we adhere to some standard notations, where, $\mathcal{N}(\mathcal{H}_{n})$ denotes the space of all signed finite \textit{Radon Measures} on $\mathcal{H}_{n}$ equipped with the norm. In other words, we identify $\mathcal{N}(\mathcal{H}_{n})$ with the dual of $C_{0}(\mathcal{H}_{n})$, the completion of all continuous functions $u:\mathcal{H}_{n}\longrightarrow \mathbb{R}$, having con=mpact support, and also is connected to the supremum norm $||.||_{\infty}$.\par 
	Important observation one can make here is that, the problem \eqref{30} has a variational structure. The \textit{Euler-Lagrange Functional}, $\mathcal{J}_{\lambda}: HW^{1,p}_{V}(\mathcal{H}_{n})\longrightarrow \mathbb{R}$ associated to this problem is defined as follows :
	\begin{align}\label{35}
		\mathcal{J}_{\lambda}(u)=\frac{1}{p}\mathcal{M}\left(\left|\left|D_{H}(u)\right|\right|^{p}_{p}+\left|\left|u\right|\right|^{p}_{p,V}\right)-\lambda \int\limits_{\mathcal{H}_{n}}F(\xi,u)d\xi -\frac{1}{p^{*}}\int\limits_{\mathcal{H}_{n}}|u|^{p^{*}}d\xi 
	\end{align}
	It implies that, under the conditions described in \eqref{34}, $\mathcal{J}_{\lambda}$ is of class $C^{1}\left(HW^{1,p}_{V}(\mathcal{H}_{n})\right)$. Moreover, for every $u,v\in HW^{1,p}_{V}(\mathcal{H}_{n})$, we define the \textit{Fr$\acute{e}$chet Derivative} of $\mathcal{J}_{\lambda}$ is given as:
	\begin{align*}
		\langle\mathcal{J}'_{\lambda}(u),v\rangle=M\left(\left|\left|D_{H}(u)\right|\right|^{p}_{p}+\left|\left|u\right|\right|^{p}_{p,V}\right)\left(\langle \mathcal{A}(u), v\rangle +\int\limits_{\mathcal{H}_{n}}V(\xi)|u|^{p-2}uvd\xi\right) \\
		\hspace{50pt} -\lambda \int\limits_{\mathcal{H}_{n}}f(\xi,u)uvd\xi-\int\limits_{\mathcal{H}_{n}}|u|^{p^{*}-2}uvd\xi
	\end{align*}
	Where, 
	\begin{align*}
		\langle\mathcal{A}_{p}(u),v\rangle=\int\limits_{\mathcal{H}_{n}}\left|D_{H}(u)\right|^{p-2}_{H}.D_{H}(u).D_{H}(v)d\xi
	\end{align*}
	\begin{rmk}
		It can be verified that the weak solutions for problem \eqref{30} indeed coincide with the \textit{critical points} of $\mathcal{J}_{\lambda}$.
	\end{rmk}
	With all these notations and definitions, we define,
	\begin{dfn}\label{dfn24}
		A sequence $\{u_{n}\}_{n\in \mathbb{N}}\subset X$ is termed as a \textbf{Palais-Smale Sequence} for the functional $\mathcal{J}_{\lambda}$ at level $c$ if,
		\begin{align}\label{36}
			\mathcal{J}_{\lambda}(u_{n})\longrightarrow c \mbox{ , }\hspace{20pt} \mathcal{J}'_{\lambda}(u_{n})\longrightarrow 0 \hspace{10pt} \mbox{ in }X' \mbox{ as } n\rightarrow +\infty 
		\end{align}
	\end{dfn}
	If \eqref{36} implies the existence of a subsequence of $\{u_{n}\}_{n}$ which converges in $X$, we assert that, $\mathcal{J}_{\lambda}$ satisfies the \textbf{Palais-Smale Condition}  $(PS)_{c}$. Moreover, if this strongly convergent subsequence exists only for some $c$ values, we comment that, $\mathcal{J}_{\lambda}$ satisfies a Local \textit{Palais-Smale Condition}. 	
	\subsection{The Non-Degenerate Case}
	In this section, we shall state and prove two theorems which best illustrates our purpose of analyzing existence and multiplicity of solutions for the problem \eqref{30} in the \textit{Non-Degenerate case}.
	\begin{thm}\label{thm6}
		A priori we assume that, \eqref{32} holds true. If $M$ satisfies conditions $(1)$ and $(2)$ of \eqref{33}, and $f$ verifies \eqref{34}, then $\exists$ \hspace{10pt} $\lambda_{1}>0$ such that, for any $\lambda\geq \lambda _{1}$, the problem \eqref{30} has a non-trivial solution in $HW^{1,p}_{V}(\mathcal{H}_{n})$.
	\end{thm}
	\begin{thm}\label{thm7}
		A priori we assume that, \eqref{32} holds true. If $M$ satisfies conditions $(1)$ and $(2)$ of \eqref{33}, and $f$ verifies \eqref{34}. Additionally, suppose we consider that, one of the following condition also holds :
		\begin{enumerate}
			\item $\exists$ a positive constant $m^{*}>0$ for every $m_{0}>m^{*}$ and $\lambda>0$.
			\item $\exists$ a positive constant $\lambda_{2}>0$ for every $\lambda>\lambda_{2}$ and, $m_{0}>0$.
		\end{enumerate}
		Then, the problem \eqref{30} admits of at least $n$ pairs of non-trivial weak solutions in $HW^{1,p}_{V}(\mathcal{H}_{n})$.
	\end{thm}
	Before we indulge ourselves into the proof of these statements, it is imperative to study all the necessary results required for justification of the same.
	\begin{thm}\label{thm8}
		(Mountain Pass Theorem)  For a given real Banach Space $E$ and $\mathcal{J}\in C^{1}(E)$, satisfying $\mathcal{J}(0)=0$. We further assume that,
		\begin{enumerate}
			\item $\exists$ $\rho, \alpha>0$ such that, $\mathcal{J}(u)\geq \alpha$ \hspace{10pt}$\forall$ $u\in E$, with $||u||_{E}=\rho$;
			\item $\exists$ $e\in E$ satisfying $||e||_{E}>\rho$ such that, $\mathcal{J}(e)<0$
		\end{enumerate}
		If we denote, $\gamma:=\left\{\gamma\in C([0,1],E)\mbox{ }|\mbox{ } \gamma(0)=1 \mbox{ , }\gamma(1)=e\right\}$. Then,
		\begin{align}
			c=\inf\limits_{\gamma\in \Gamma}\max\limits_{0\leq t\leq 1}\mathcal{J}(\gamma(t)) \geq \alpha
		\end{align}
		and there exists a $(PS)_{c}$ sequence $\{u_{n}\}_{n} \subset E$.
	\end{thm}
	One can in fact verify that, $\mathcal{J}_{\lambda}$ satisfies the geometric properties $(1)$ and $(2)$ of \textit{Mountain Pass Theorem} \eqref{thm8}. 
	\subsubsection{Proof of Theorem \eqref{thm6}}
	First, we intend to establish the following:
	\begin{claim}\label{37}
		We shall have,
		\begin{align}
			0<c_{\lambda}=\inf\limits_{\gamma\in \Gamma}\max\limits_{0\leq t\leq 1}\mathcal{J}_{\lambda}(\gamma(t))<\left(\frac{1}{\theta}-\frac{1}{p^{*}}\right)\left(m_{0}C_{p^{*}}\right)^{\frac{p^{*}}{p^{*}-p}}
		\end{align}
		for large enough $\lambda$. 
	\end{claim}
	\begin{proof}
		Choose $v_{0}\in HW^{1,p}_{V}(\mathcal{H}_{n})$ such that,
		\begin{align*}
			||v_{0}||=1 \mbox{ , and, }\hspace{20pt}\lim\limits_{t\rightarrow \infty}\mathcal{J}_{\lambda}(tv_{0})=-\infty
		\end{align*}
		Therefore, $\sup\limits_{t\geq 0}\mathcal{J}_{\lambda}(tv_{0})=\mathcal{J}_{\lambda}(t_{\lambda}v_{0})$ for some $t_{\lambda}>0$. Thus, $t_{\lambda}$ satisfies,
		\begin{align*}
			M\left(\left|\left|D_{H}(t_{\lambda}v_{0})\right|\right|^{p}_{p}+\left|\left|t_{\lambda}v_{0}\right|\right|^{p}_{p,V}\right).\left(\left|\left|D_{H}(t_{\lambda}v_{0})\right|\right|^{p}_{p}+\left|\left|t_{\lambda}v_{0}\right|\right|^{p}_{p,V}\right)
		\end{align*}
		\begin{align}\label{38}
			=\lambda \int\limits_{\mathcal{H}_{n}}f(\xi,t_{\lambda}v_{0})|t_{\lambda}v_{0}|^{2}d\xi + \int\limits_{\mathcal{H}_{n}}|t_{\lambda}v_{0}|^{p^{*}}d\xi
		\end{align}
		It suffices to prove that, $\{t_{\lambda}\}_{\lambda>0}$ is bounded. Since, $t_{\lambda}\geq 1$ $\forall$  $\lambda>0$. Thus, by property $(2)$ of \eqref{33} and \eqref{38}, we deuce that,
		\begin{center}
			$	\tau \mathcal{M}(1)t^{2p\tau}_{\lambda}\geq \tau \mathcal{M}(1)(||t_{\lambda}v_{0}||^{p})^{\tau}$
		\end{center}
		\begin{align*}
			\geq 	M\left(\left|\left|D_{H}(t_{\lambda}v_{0})\right|\right|^{p}_{p}+\left|\left|t_{\lambda}v_{0}\right|\right|^{p}_{p,V}\right).\left(\left|\left|D_{H}(t_{\lambda}v_{0})\right|\right|^{p}_{p}+\left|\left|t_{\lambda}v_{0}\right|\right|^{p}_{p,V}\right) \\
			=\lambda \int\limits_{\mathcal{H}_{n}}f(\xi,t_{\lambda}v_{0})|t_{\lambda}v_{0}|^{2}d\xi + \int\limits_{\mathcal{H}_{n}}|t_{\lambda}v_{0}|^{p^{*}}d\xi
		\end{align*}
		\begin{align}\label{39}
			\geq t^{p^{*}}_{\lambda}\int\limits_{\mathcal{H}_{n}}|v_{0}|^{p^{*}}d\xi 
		\end{align}
		It follows from \eqref{39} that, $\{t_{\lambda}\}_{\lambda>0}$ is bounded, since, $p\tau<p^{*}$.
		Next, we shall prove that, $t_{\lambda}\longrightarrow 0$ as, $\lambda\longrightarrow \infty$. By contradiction, suppose, $\exists$ $t_{0}>0$ and a sequence $\{\lambda_{n}\}_{n}$ with $\lambda_{n}\longrightarrow \infty$ as $n\rightarrow \infty$, and consequently, $t_{\lambda_{n}}\longrightarrow t_{0}$ as $n\rightarrow \infty$.\par 
		A simple application of the \textit{Lebesgue Dominated Convergence Theorem} yields,
		\begin{align*}
			\int\limits_{\mathcal{H}_{n}}f(\xi, t_{\lambda_{n}}v_{0})\left|t_{\lambda_{n}}v_{0}\right|^{2}d\xi \longrightarrow \int\limits_{\mathcal{H}_{n}}f(\xi, t_{\lambda}v_{0})\left|t_{\lambda}v_{0}\right|^{2}d\xi 
		\end{align*}
		as $n\rightarrow \infty$. Thus, we can conclude that, 
		\begin{align*}
			\lambda_{n}\int\limits_{\mathcal{H}_{n}}f(\xi, t_{\lambda}v_{0})\left|t_{\lambda}v_{0}\right|^{2}d\xi \longrightarrow \infty \mbox{ , as, }\hspace{10pt} n\rightarrow \infty
		\end{align*}
		A contradiction to the fact that, \eqref{38} holds true. Thus, $t_{\lambda}\longrightarrow 0$ as, $\lambda\rightarrow \infty$. Another deduction which can indeed be made from \eqref{38} is,
		\begin{align*}
			\lim\limits_{\lambda\rightarrow \infty}\lambda \int\limits_{\mathcal{H}_{n}}f(\xi, t_{\lambda}v_{0})\left|t_{\lambda}v_{0}\right|^{2}d\xi =0
		\end{align*}
		and,
		\begin{align*}
			\lim\limits_{\lambda\rightarrow \infty}\int\limits_{\mathcal{H}_{n}}\left|t_{\lambda}v_{0}\right|^{p^{*}}d\xi =0
		\end{align*}
		A priori, using the deduction, $t_{\lambda}\longrightarrow 0$ as, $\lambda\rightarrow \infty$ and the definition of $\mathcal{J}_{\lambda}$, we obtain,
		\begin{align*}
			\lim\limits_{\lambda\rightarrow \infty}\left(\sup\limits_{t\geq 0}\mathcal{J}_{\lambda}(tv_{0})\right)=	\lim\limits_{\lambda\rightarrow \infty}\mathcal{J}_{\lambda}(tv_{0})=0
		\end{align*}
		Hence, $\exists$ $\lambda_{1}>0$ satisfying for every $\lambda\geq \lambda_{1}$,
		\begin{align*}
			\sup\limits_{t\geq 0}\mathcal{J}_{\lambda}(tv_{0})<\left(\frac{1}{\theta}-\frac{1}{p^{*}}\right)\left(m_{0}C_{p^{*}}\right)^{\frac{p^{*}}{p^{*}-p}}
		\end{align*}
		Choosing $e=t_{1}v_{0}$, for large enough $t_{1}$, it can be verified that, 
		\begin{center}
			$\mathcal{J}_{\lambda}(e)<0$
		\end{center}
		\begin{align*}
			\implies 0<c_{\lambda}\leq \max\limits_{t\in [0,1]}\mathcal{J}_{\lambda}(\gamma(t))\hspace{40pt} \mbox{, by setting, }\gamma(t)=tt_{1}v_{0}
		\end{align*}
		Therefore, we conclude,
		\begin{align*}
			0<c_{\lambda}=\sup\limits_{t\geq 0}\mathcal{J}_{\lambda}(tv_{0})<\left(\frac{1}{\theta}-\frac{1}{p^{*}}\right)\left(m_{0}C_{p^{*}}\right)^{\frac{p^{*}}{p^{*}-p}}
		\end{align*}
		for $\lambda$ large enough, which completes the proof.
	\end{proof}
	\subsubsection{Proof of Theorem \eqref{thm7}}
	To prove the theorem, we shall in fact use the concept related to \textit{Krasnoselskii}'s Genus Theory \cite{52}. Given a \textit{Banach Space} $X$, and let, $\Lambda$ denotes the class of all \textit{closed subsets} $A\subset X\setminus \{0\}$ that are symmetric with respect to the origin ( in other words, $u\in A$ implies, $-u\in A$ ).
	\begin{thm}\label{thm9}
		For an infinite dimensional Banach Space $X$ and $\mathcal{J}\in C^{1}(X)$ be an even functional, such that, $\mathcal{J}(0)=0$. Further, we assume that, $X=Y\oplus Z$, $Y$ being finite dimensional. Then, $\mathcal{J}$ satisfies the following conditions,
		\begin{enumerate}
			\item $\exists$ constants $\rho, \alpha>0$ satisfying, $\mathcal{J}(u)\geq \alpha$ for every $u\in \partial B_{\rho}\cap Z$.
			\item $\exists$ $\Theta>0$ such that, $\mathcal{J}$ satisfies the $(PS)_{c}$ condition $\forall$ $c$, with $c\in (0,\Theta)$.
			\item Given any finite dimensional subspace $\tilde{X}\subset X$, $\exists$ $R=R(\tilde{X})>0$ satisfying $\mathcal{J}(u)\leq 0$ on $\tilde{X}\setminus B_{R}$
		\end{enumerate}
	\end{thm}
	In addition to the above, we further consider, $dim(Y)=k$  and that, $Y=Span\{v_{1},v_{2},\cdots ,v_{k}\}$. For $n\geq k$, we inductively choose that, $v_{n+1}\notin E_{n}=Span\{v_{1},v_{2},\cdots ,v_{n}\}$. Suppose, $R_{n}=R(E_{n})$, and, $\Omega_{n}=B_{R_{n}}\cap E_{n}$. Define, 
	\begin{align*}
		G_{n}:=\left\{\psi\in C(\Omega_{n},X)\mbox{ }|\mbox{ } \psi|_{\partial B_{R_{n}\cap E_{n}}}=id \mbox{ and, }\psi \mbox{  is odd}
		\right\}
	\end{align*}
	and,
	\begin{align*}
		G_{n}:=\left\{\psi(\overline{\Omega_{n}\setminus V})\mbox{ }|\mbox{ } \psi\in G_{n}\mbox{ , }n\geq j \mbox{ , }V\in \Lambda \mbox{ , }\gamma(V)\leq n-j
		\right\}
	\end{align*}
	Where, $0\leq c_{j}\leq c_{j+1}$ and, $c_{j}<\Theta$ for $j>k$, then, we assert that, $c_{j}$ is a \textit{critical value }of $\mathcal{J}$.\par 
	Furthermore, if $c_{j}=c_{j+l}=c<\Theta$ for every $1\leq l\leq m$ and for $j>k$. Thus, $\gamma(K_{c})\geq m+1$, provided,
	\begin{align*}
		K_{c}:=\left\{u\in X \mbox{ }|\mbox{ }\mathcal{J}(u)=c \mbox{ and, }\mathcal{J}'(u)=0\right\}
	\end{align*}
	Now that, we have introduced all the terminologies required for the proof of Theorem \eqref{thm7}, we proceed to the proof of the statement of the same.
	\begin{proof}
		We shall apply Theorem \eqref{thm9} to $\mathcal{J}_{\lambda}$. A priori using the fact that, $HW^{1,p}_{V}(\mathcal{H}_{n})$ is indeed a \textbf{reflexive} \textit{Banach Space} and, $\mathcal{J}_{\lambda}\in C^{1}\left(HW^{1,p}_{V}(\mathcal{H}_{n})\right)$.\par 
		Applying \eqref{35}, we can infer that, the functional $\mathcal{J}_{\lambda}$ satisfies, $\mathcal{J}_{\lambda}(0)=0$. We shall thus prove the theorem in three steps.\\
		\textbf{\underline{Step 1 :}} \par 
		We can in fact obtain that, $\mathcal{J}_{\lambda}$ indeed satisfies $(1)$ and $(3)$ of Theorem \eqref{thm9}.\\
		\textbf{\underline{Step 2 :}} \par 
		We claim that, $\exists$ a monotone increasing sequeence $\{\Phi_{n}\}_{n}$ in $\mathbb{R}_{+}$, such that, 
		\begin{align*}
			c^{\lambda}_{n}=\inf\limits_{E\in \Gamma_{n}}\max\limits_{u\in E}\mathcal{J}_{\lambda}(u)<\Phi_{n}
		\end{align*} 
		The above definition of $	c^{\lambda}_{n}$ allows us to infer, 
		\begin{align*}
			c^{\lambda}_{n}=\inf\limits_{E\in \Gamma_{n}}\max\limits_{u\in E}\mathcal{J}_{\lambda}(u)\leq \inf\limits_{E\in \Gamma_{n}}\max\limits_{u\in E}\left\{\mathcal{M}\left(\left|\left|D_{H}(u)\right|\right|^{p}_{p}+||u||^{p}_{p,V}\right)-\frac{1}{p^{*}}\int\limits_{\mathcal{H}_{n}}|u|^{p^{*}}d\xi\right\}
		\end{align*}
		Set,
		\begin{align*}
			\Phi_{n}=\inf\limits_{E\in \Gamma_{n}}\max\limits_{u\in E}\left\{\mathcal{M}\left(\left|\left|D_{H}(u)\right|\right|^{p}_{p}+||u||^{p}_{p,V}\right)-\frac{1}{p^{*}}\int\limits_{\mathcal{H}_{n}}|u|^{p^{*}}d\xi\right\}
		\end{align*}
		Therefore, $\Phi_{n}<\infty$ and, $\Phi_{n}\leq \Phi_{n+1}$ \hspace{10pt}$\forall$ $n\in \mathbb{N}$ ( Follows from the definition of $\Gamma_{n}$ ).\\
		\textbf{\underline{Step 3 :}} \par 
		We claim that, the problem \eqref{30} has at least $k$ pairs of weak solutions. To this end, we distingish two cases:\\ 
		\textit{Case I :}\par 
		Fix $\lambda >0$. We set $m_{0}$ very large such that, 
		\begin{align*}
			\sup\limits_{n}\Phi_{n}<\left(\frac{1}{\theta}-\frac{1}{p^{*}}\right)\left(m_{0}C_{p^{*}}\right)^{\frac{p^{*}}{p^{*}-p}}
		\end{align*}
		\textit{Case II :}\par 
		using similar discussions as in Theorem \eqref{thm9}, $\exists$ $\lambda_{2}>0$ satisfyng,
		\begin{align*}
			c^{\lambda}_{n}\leq \Phi_{n}<\left(\frac{1}{\theta}-\frac{1}{p^{*}}\right)\left(m_{0}C_{p^{*}}\right)^{\frac{p^{*}}{p^{*}-p}} \hspace{40pt}\forall \mbox{ }\lambda>\lambda_{2}.
		\end{align*}
		Thus, in any case, we must obtain,
		\begin{align*}
			0<c^{\lambda}_{1}\leq c^{\lambda}_{2}\leq \cdots \leq c^{\lambda}_{n}<\Phi_{n}<\left(\frac{1}{\theta}-\frac{1}{p^{*}}\right)\left(m_{0}C_{p^{*}}\right)^{\frac{p^{*}}{p^{*}-p}}
		\end{align*}
		Applying similar arguments as described in \cite{52}, we can in fact guarantee that, the levels, $c^{\lambda}_{1}\leq c^{\lambda}_{2}\leq \cdots \leq c^{\lambda}_{n}$ are indeed the \textit{critical values} of $\mathcal{J}_{\lambda}$.\par 
		If $c^{\lambda}_{j}=c^{\lambda}_{j+1}$ for some $j=1,2,\cdots k-1$, therefore, by \cite[Theorem~$4.2$, Remark~$2.12$]{53}, the set $K_{c^{\lambda}_{j}}$ contains infinitely many distinct points, and hence, the problem \eqref{30} has \textbf{infinitely many weak solutions}. \par 
		Finally thus we conclude that, the problem \eqref{30} has at least $k$ pairs of solutions, and the proof is done.
		
	\end{proof}
	\subsection{The Degenerate Case}
	In this section, we shall investigate the degradation of the problem \eqref{30}. We shall always assume that, $M$ satisfies conditions $(2)$ and $(3)$ of \eqref{33}, and $f$ verifies the coonditions mentioned in \eqref{34}. We state the main results which helps us comment on the existence and the multiplicity of the solutions to problem \eqref{30} in the \textit{Degenerate Case}.
	\begin{thm}\label{thm10}
		A priori we assume that, \eqref{32} holds true. If $M$ satisfies conditions $(2)$ and $(3)$ of \eqref{33}, and $f$ verifies \eqref{34}, then $\exists$ \hspace{10pt} $\lambda_{3}>0$ such that, for any $\lambda\geq \lambda _{3}$, the problem \eqref{30} has a non-trivial solution in $HW^{1,p}_{V}(\mathcal{H}_{n})$.
	\end{thm}
	\begin{thm}\label{thm11}
		A priori we assume that, \eqref{32} holds true. If $M$ satisfies conditions $(2)$ and $(3)$ of \eqref{33}, and $f$ verifies \eqref{34}. Additionally, suppose we consider that, one of the following condition also holds :
		\begin{enumerate}
			\item $\exists$ a positive constant $m_{*}>0$ for every $m_{1}>m_{*}$ and $\lambda>0$.
			\item $\exists$ a positive constant $\lambda_{4}>0$ for every $\lambda>\lambda_{4}$ and, $m_{1}>0$.
		\end{enumerate}
		Then, the problem \eqref{30} admits of at least $n$ pairs of non-trivial weak solutions in $HW^{1,p}_{V}(\mathcal{H}_{n})$.
	\end{thm}
	Before indulging ourselves into the proofs of these statements above, we must state the following lemma which apparently plays a major role in determining the existence of solution to problem \eqref{30}.
	\begin{lemma}\label{40}
		Under the assumptions on $M$ and $f$ mentioned above, if suppose, $\{u_{n}\}_{n}\subset HW^{1,p}_{V}(\mathcal{H}_{n})$ be a Palais-Smale Sequence of functional $\mathcal{J}_{\lambda}$ ( as defined in \eqref{36} ). \par 
		Furthermore, if, 
		\begin{align*}
			0<c_{\lambda}<\left(\frac{1}{\theta}-\frac{1}{p^{*}}\right)\left(m_{0}C_{p^{*}}\right)^{\frac{p^{*}}{p^{*}-p}}
		\end{align*}
		Then, $\exists$ a subsequence of $\{u_{n}\}_{n}$ \textbf{Strongly Convergent} in $HW^{1,p}_{V}(\mathcal{H}_{n})$.
	\end{lemma}
	For the proof of the Lemma, interested readers can look up \cite{51}.\par 
	\begin{rmk}
		In fact, one can further infer that, the functional $\mathcal{J}_{\lambda}$ satisfies all the assumptions of the \textbf{Mountain Pass Theorem} ( Theorem \eqref{thm8} ).
	\end{rmk}
	\subsubsection{Proof of Theorem \eqref{thm10}}
	\begin{proof}
		A priori using similar argments as in the proof of the \textit{Mountain Pass Theorem} (Theorem \eqref{thm8}), we get,
		\begin{align*}
			0<c_{\lambda}=\inf\limits_{\gamma\in \Gamma}\max\limits_{0\leq t\leq 1}\mathcal{J}_{\lambda}(\gamma(t))<\left(\frac{1}{\theta}-\frac{1}{p^{*}}\right)\left(m_{0}C^{\tau}_{p^{*}}\right)^{\frac{p^{*}}{p^{*}-p\tau}}
		\end{align*}
		The rest of the proof follows in a similar manner to that described in Theorem \eqref{thm6}.
	\end{proof}
	\subsubsection{Proof of Theorem \eqref{thm11}}
	\begin{proof}
		The proof of the theorem is similar to that of Theorem \eqref{thm7}.
	\end{proof}
	
	\vspace{40pt}
	\section*{Statements and Declarations}
	\subsection*{Conflicts of Interest Statement}
	I as the sole author of this article certify that I have no affiliations with or involvement in any
	organization or entity with any financial interest (such as honoraria; educational grants; participation in speakers’ bureaus;
	membership, employment, consultancies, stock ownership, or other equity interest; and expert testimony or patent-licensing
	arrangements), or non-financial interest (such as personal or professional relationships, affi liations, knowledge or beliefs) in
	the subject matter or materials discussed in this manuscript.\\
	\subsection*{Data Availability Statement}
	I as the sole author of this aarticle confirm that the data supporting the findings of this study are available within the article [and/or] its supplementary materials.\\\\

\end{document}